\documentclass[10pt]{article}
\usepackage{amsmath}
\usepackage[arrow, matrix, curve]{xy}
\usepackage{amssymb}
\usepackage{a4wide}
\usepackage{amsthm}
\newtheorem{theorem}{Theorem}[section]
\newtheorem{lemma}[theorem]{Lemma}

\newtheorem{proposition}[theorem]{Proposition}
\theoremstyle{definition}

\theoremstyle{remark}
\newtheorem{remark}[theorem]{Remark}

\numberwithin{equation}{section}

\DeclareMathOperator{\Gl}{GL}
\DeclareMathOperator{\su}{SU}
\DeclareMathOperator{\Sl}{SL}

\newcommand{\CC}{\mathbb C}
\newcommand{\BB}{\mathbb B}
\newcommand{\PP}{\mathbb P}
\newcommand{\QQ}{\mathbb Q}
\newcommand{\RR}{\mathbb R}
\newcommand{\LL}{\mathbb L}

\newcommand{\NN}{\mathbb N}
\newcommand{\ZZ}{\mathbb Z}
\newcommand{\G}{\mathbf G}
\newcommand{\Pp}{\mathbf P}

\newcommand{\dd}{{\mathbf d}}

\newcommand{\aaa}{\mathfrak a}
\newcommand{\pP}{\mathfrak p}

\newcommand{\OO}{\mathcal O}
\begin{document}
\title{Invariants of some compactified Picard modular surfaces and applications}
\author{Amir D\v{z}ambi\'c\\ \\ Institut f\"ur Mathematik\\ Goethe-Universit\"at Frankfurt\\ 
Robert-Mayer Str.~6-8, 60325 Frankfurt am Main\\ Email: dzambic@math.uni-frankfurt.de}
\date{}
\maketitle
\begin{abstract}
The paper investigates invariants of compactified Picard modular surfaces by principal congruence subgroups of Picard modular groups. The applications to the surface classification and modular forms are discussed.
\end{abstract}

\section{Introduction}
The following note should be considered as a supplement to the work of R.~P.~Holzapfel on invariants of Picard modular surfaces which are quotients of the two dimensional complex ball by principal congruence subgroups.\\ In his papers \cite{Ho:min} and \cite{Ho:zetadim}, Holzapfel developes concrete formulas for the Chern numbers and related invariants of compactified ball quotients by principal congruence subgroups $\Gamma_K(N)$ of Picard modular groups $\Gamma_K=\su((2,1),\OO_K)$ where $K=\QQ(\sqrt{-d})$ is an imaginary quadratic field and $N$ is a positive integer. There, he uses results which are mainly developed in \cite{Ho:bsa}. Using Riemann-Roch theory in combination with proportionality theorems, he also gets informations on the classification of these compactified ball quotients and dimensions of spaces of cusp forms relative to the congruence subgroups.\\
In this article we slightly extend Holzapfel's results considering not only principal congruence subgroups by natural numbers, but also general integral ideals $\aaa$ of the quadratic field $K=\QQ(\sqrt{-d})$. Including some technical number theoretical details, Holzapfel's arguments can also be applied to this slightly more general and larger class of congruence subgroups. One technical result is a formula for the index $[\Gamma_K:\Gamma_K(\aaa)]$. The final results in principle do not differ from those obtained in \cite{Ho:min} und \cite{Ho:zetadim}: We characterize the class of principal congruence subgroups -- excluding few technically caused possible ``exceptional cases'' -- for which the (smooth) compactified ball quotient is a surface of general type. For these surfaces the ``coordinates'' $(c_2,c_1^2)$ in the surface geography are explicitly known. The dimensions of spaces of cusp forms can also be computed explicitely.\\

\section{Picard modular groups and their congruence subgroups}
Let $K=\QQ(\sqrt{-d})$ be an imaginary quadratic field, and denote $\OO_K$ the ring of integers of $K$. Let $V$ be a 3-dimensional $K$-vector space equipped with a hermitian form $h:V\times V\longrightarrow K$. Let us assume that the signature of $h$ is $(2,1)$, i.~e.~that $h$ has two positive and one negative eigenvalue. Such $h$ can, for instance be represented by the diagonal matrix $\textrm{diag}(1,1,-1)$, but often it is useful to use another hermitian form.
We consider the special unitary group 
\[\G:=\su(h)=\{g\in \Sl_3(\CC)\mid h(gv,gw)=h(v,w)\ \text{for all} v,w\in V_{\CC}\}\] 
as an algebraic group defined over $\QQ$. Its group of $\QQ$-rational points is 
\[\G(\QQ):=\su(h,K)=\{g\in \Sl_3(K)\mid h(gv,gw)=h(v,w),\ \textrm{for\ all}\ v,w\in V\}.\]
Often the interpretation of $\G$ as a group corresponding to an involution is useful. For this, we remark that the map $\iota=\iota_h$ which associates with every matrix $g\in M_3(K)$ the matrix $g^{\iota_h}:=M_h \bar g^{t} M_h^{-1}$, $M_h$ denoting the matrix which represents $h$ with respect to a suitable basis ($\textrm{diag}(1,1,-1)$ say), defines an involution of second kind on the matrix algebra $M_3(K)$. This means that $\iota$ is an anti-automorphism of $M_3(K)$ which acts as the complex conjugation, when restricted to the diagonal matrices $\textrm{diag}(\alpha,\alpha,\alpha)$. In this description the group $\G$ appears as the groups of all matrices $g\in \Sl_3$ such that $gg^{\iota_h}=1_3$.\\
 
We define the so-called {\bf full Picard modular group} as
\begin{equation*}
 \Gamma_{K}=\su(h,\OO_K)=\Sl_3(\OO_K)\cap \G(\QQ).
\end{equation*}
$\Gamma_K$ is obviously an arithmetic subgroup in $\G(\QQ)$ as well as its subgroups of finite index which we simply call {\bf Picard modular groups}. If $\mathfrak a$ is an ideal of $\OO_K$, let $\Gamma_K(\mathfrak a)$ denote the {\bf principal congruence subgroup} of $\Gamma_K$ with respect to $\mathfrak a$. It is defined as the subgroup of $\Gamma_K$ consisting of all elements $\gamma\in \Gamma_K$ such that $\gamma-1_3\in M_3(\aaa)$. In other words, $\Gamma_K(\aaa)$ can be defined as the kernel of the canonical reduction map 
\[\rho: \Gamma_K \longrightarrow \Sl_3(\OO_K/\aaa).\]
We define the {\bf level} of $\Gamma_K(\mathfrak a)$ as the absolute norm $N(\mathfrak a)=|\OO_K/\aaa|$. There are finitely many principal congruence subgroups of fixed level. By definition a {\bf congruence subgroup} of $\Gamma_K$ is a group which contains a principal congruence subgroup. Since every principal congruence subgroup has a finite index in $\Gamma_K$ (which is obvious by definition), all congruence subgroups are Picard modular groups. In this section we will treat the following two technical problems:
\begin{itemize}
\item Compute the index $[\Gamma_K:\Gamma_K(\aaa)]$.
\item Determine $\aaa$ for which the principal congruence subgroup $\Gamma_K(\aaa)$ is a neat subgroup.
\end{itemize}
Thereby we say that an arithmetic group is {\bf neat}, if the subgroup of $\CC^{\ast}$ generated by the eigenvalues of all elements in $\Gamma$ is torsion free. A neat group is obviously torsion free. 
\subsection{Index computations}
In order to compute the index $[\Gamma_K:\Gamma_K(\aaa)]$, we make use of some local-to-global principles which are applicable in the present case.\\
For a prime number $p$ the group of $\QQ_p$-rational points $\G(\QQ_p)$ is 
\begin{equation}
\label{gqp}
\G(\QQ_p)=\{g\in \Sl_3(K\otimes_{\QQ} \QQ_p)\mid gg^{\iota_p}=1_3 \},
\end{equation}
where we write $\iota_p$ for the natural extension of $\iota_h$ to the algebra $M_3(K\otimes_{\QQ} \QQ_p)$. The group $\G(\ZZ_p)$ is defined in an obvious manner. We can also define principal congruence subgroups $\G(\ZZ_p)(\aaa)$ of $\G(\ZZ_p)$ in the same way as in the case $\Gamma_K$, now taking two-sided ideals $\aaa$ of the order $\OO_K\otimes \ZZ_p$ into account. The following lemma is the starting point of the index calculation. It is an important consequence of the strong approximation.
\begin{lemma}
\label{strongappr}
Let $\aaa\subset \OO_K$ be an integral ideal of $K$ and let $\aaa=\pP_1^{e_1}\cdots\pP_t^{e_t}$ be the prime ideal decomposition of $\aaa$, and denote $p_i=\pP_i\cap\ZZ$ the unique integer prime over which $\pP$ lies. Then, there is an isomorphism
\[\Gamma_K/\Gamma_K(\aaa)\cong \prod_{i=1}^{t}\G(\ZZ_{p_i})/\G(\ZZ_{p_i})(\pP^{e_i}).\]
In particular 
\[[\Gamma_K:\Gamma_K(\aaa)]=\prod_{i=1}^{t}[\G(\ZZ_{p_i}):\G(\ZZ_{p_i})(\pP^{e_i})]\]
\end{lemma}
\begin{proof}
Let $\lambda$ be the map
\[\lambda: \Gamma_K\longrightarrow \prod_{i=1}^{t}\G(\ZZ_{p_i})/\G(\ZZ_{p_i})(\pP^{e_i})\]
given by
\[\gamma\mapsto (\gamma\bmod \pP_1^{e_1},\ldots,\gamma\bmod \pP_t^{e_t}).\] 
The kernel of $\lambda$ is obviously $\Gamma_K(\aaa)$, by definition. To show the surjectivity we recall the strong approximation property: For any given finite set of primes, $p_1,\dots,p_t$ say, and given elements $g_{p_1}\in\G(\ZZ_{p_1}),\ldots,g_{p_t}\in\G(\ZZ_{p_t})$ as well as exponents $e_1,\ldots,e_t$ one can find an element $g\in \G(\ZZ)=\Gamma_K$ such that $g\equiv g_{p_i}\bmod p_i^{e_i}$ for $i=1,\ldots,t$. But as $\pP_i|p_i$ this $g$ satisfies $g\equiv g_{p_i}\bmod \pP_i^{e_i}$. This already proves the surjectivity of $\lambda$.
\end{proof}
In order to compute the local indices, we first need to know more about the local groups $\G(\QQ_p)$. Looking at their definition (\ref{gqp}) we see that their structure highly depends on the structure of $K\otimes \QQ_p$, which itself is determined by the decomposition behaviour of the prime $p$. Essentially there are two cases:
\begin{enumerate}
\item[i)] $p$ is decomposed in $K=\QQ(\sqrt{-d})$. In this case $K\otimes \QQ_p\cong K_{\pP}\oplus K_{\bar\pP}$ with two conjugate prime ideals $\pP\neq \bar \pP$ in $\OO_K$ such that $p=\pP\bar \pP$. Therefore $M_3(K\otimes \QQ_p)\cong M_3(K_{\pP})\oplus M_3(K_{\bar \pP})$. Since $p$ is decomposed in $K$, $\left ( \frac{-d}{p}\right )=1$ and $-d$ is a square in $\QQ_p$. Therefore $K_{\pP}\cong K_{\bar \pP}\cong\QQ_p$. The extension of the field automorphism $\bar \cdot: K\longrightarrow K$ to $K\otimes \QQ_p$ must be an involution on $K_{\pP}\oplus K_{\bar \pP}\cong \QQ_p\oplus \QQ_p$ and the only possibility is $\overline{(x,y)}=(y,x)$. Therefore the extension $\iota_p$ of $\iota=\iota_h$ is given by changing the summands in $M_3(\QQ_p)\oplus M_3(\QQ_p)$. Now it is easy to see that the projection on one of the summands gives an isomorphy between $\G(\QQ_p)$ and $SL_3(\QQ_p)$, since $\G(\QQ_p)$ is defined as the group consisting of those pairs $(g,h)\in \Sl_3(\QQ_p)\oplus \Sl_3(\QQ_p)$ with $(g^{-1},h^{-1})=(h,g)$.

\item[ii)] $p$ is non-decomposed in $K$. In this case there is the unique prime ideal $\pP$ lying over $p$ such that $K\otimes \QQ_p=K_{\pP}$, which is a quadratic extension of $\QQ_p$. The extended involution $\iota_p$ is an involution of second kind on $M_3(K_{\pP})$ given by $g\mapsto M_h \bar g M_h^{-1}$ and $\G(\QQ_p)=\su(h,K_{\pP})$. 
\end{enumerate}

\begin{lemma}
\label{localisom}
If $p$ is non-decomposed, $\G(\QQ_p)\cong \su_3(K_{\pP})$.
\end{lemma}
\begin{proof}
By Landherr's Theorem, the isometry class of a hermitian form $h$ in $n$ variables associated with a quadratic extension $E/F$ of local fields is uniquely determined by its discriminant $\dd(h)=[\det(h)]\in F^{\ast}/N_{E/F}(E^{\ast})\cong \ZZ/2\ZZ$. On the other hand the isomorphism class of the associated unitary group $SU(h)$ only depends on $h$ up to multiplicative constant. Hence, the unitary group associated with $\textrm{diag}(1,1,-1)$ is isomorphic to the unitary group associated with $\textrm{diag}(-1,-1,1)$. Since the discriminant of the latter is $1$, the result follows.     
\end{proof}

Now we compute the local indices

\begin{lemma}
\label{locind1}
Let $\pP$ be a prime ideal of $K=\QQ(\sqrt{-d})$ such that $p=\pP\cap \ZZ$ is decomposed in $K$. Then
\[ [\G(\ZZ_p):\G(\ZZ_p)(\pP^n)]=p^{8n}(1-p^{-3})(1-p^{-2})\]
for any positive integer $n$.
\end{lemma}
\begin{proof}
By our discussion above, in the decomposed case $\G(\QQ_p)=\Sl_3(\ZZ_p)$ and $\G(\ZZ_p)(\pP^n)=\Sl_3(\ZZ_p)(p^n)$. We note that the sequence 
\begin{equation}
 \label{seq}
1\longrightarrow SL_n(\ZZ_p)(p^n)\longrightarrow SL_n(\ZZ_p)\longrightarrow SL_n(\ZZ_p/p^n\ZZ_p)\longrightarrow 1
\end{equation}
is exact. In fact, (\ref{seq}) is exact even if we replace $\ZZ_p$ and $p^n$ by any commutative ring $R$ and ideal $\mathfrak I$ such that $R/\mathfrak I$ is finite (see \cite{bass}, Cor.~5.2). Therefore 
\[[\G(\ZZ_p):\G(\ZZ_p)(\pP^n)]=|SL_3(\ZZ_p/p^n\ZZ_p)|=|SL_3(\ZZ/p^n\ZZ)|.\] 
The latter number can be computed in an elementary way:\\
Denote $\rho: \Gl_3(\ZZ/p^n\ZZ)\longrightarrow \Gl_3(\ZZ/p\ZZ)$ the canonical projection map which sends every residue $\bmod\ p^n$ to its residue $\bmod\ p$. The kernel $\ker(\rho)$ consists of $p^{9(n-1)}$ elements. With the well-known number $|\Gl_3(\ZZ/p\ZZ)|=p^9(1-p^{-3})(1-p^{-2})(1-p^{-1})$, we get $|\Gl_3(\ZZ/p^n\ZZ)|=p^{9n}(1-p^{-3})(1-p^{-2})(1-p^{-1})$. Dividing this number by $\varphi(p^n)$ which is the order of the multiplicative group $(\ZZ/p^n\ZZ)^{\ast}$ we get the desired result.
\end{proof}
For some technical reasons we exclude from now on the prime 2 from our consideration.

\begin{lemma}
\label{locind2}
Let $\pP$ be a prime ideal in $\OO_K$ and $p=\pP\cap \ZZ\neq 2$ non-decomposed. For every $n \geq 1$ the following holds:
\begin{enumerate}
\item If $p$ is inert,
\[[\G(\ZZ_p):\G(\ZZ_p)(\pP^n)]=p^{8n}(1+p^{-3})(1-p^{-2}).\]
\item If $p$ is ramified, 
\[[\G(\ZZ_p):\G(\ZZ_p)(\pP^n)]=p^{\epsilon_n}(1-p^{-2}),\]
where $\epsilon_n$ is defined as 
\[\epsilon_n=\left\{\begin{array}{cc} 4n & \textrm{if}\ n\equiv 0\bmod 2\\ 4n-1 & \textrm{if}\ n\equiv 1\bmod 2 \end{array} \right .\]
\end{enumerate}

\end{lemma}
\begin{proof}
First, we identify the asked index with the order of the finite group $\su(\bar h,\OO_{\pP}/\pP^{n})$, where $\OO_{\pP}=\OO_{K_{\pP}}$ denotes the ring of integers in $K_{\pP}$ and $\bar h$ the restriction of $h$ to $\OO_{\pP}/\pP^n$. To do so, we have to prove the surjectivity of the canonical reduction map $\G(\ZZ_p)=\su(h,\OO_{\pP})\longrightarrow \su(h,\OO_{\pP}/\pP^n)$. But this follows from the identity 
\[\G(\ZZ_p)/\G(\ZZ_p)(p^n)\cong \su(\bar h,\OO_{\pP}/p^n),\] see \cite[Prop.~5A.2.7]{Ho:bsa}, and the surjectivity of the reduction $\su(\bar h,\OO_{\pP}/p^{n})\longrightarrow \su(\bar h,\OO_{\pP}/\pP^{n})$.\\ 
Moreover, we can extract the above local index formulas again from \cite[Cor.~5A.1.3]{Ho:bsa} for the case in which $p$ is inert, or $p$ is ramified and $n$ is an even positive integer. According to this, the only non trivial case is the situation where $p$ is ramified and $n$ is odd. Let us from now on assume this situation, that is, $p$ is ramified and $n$ is odd.\\By \cite[Cor.~5A.1.3]{Ho:bsa} again we know that \[ \G(\ZZ_p)/\G(\ZZ_p)({\pP}^{n+1})=|\G(\ZZ_p)/\G(\ZZ_p)(p^{(n+1)/2})|=p^{4(n+1)}(1-p^{-2}).\] 
By the multiplicativity of index we obtain 
\begin{equation}
\label{pnpn+1}
\G(\ZZ_p)/\G(\ZZ_p)({\pP}^n)=p^{4(n+1)}(1-p^{-2})/[\G(\ZZ_p)({\pP}^{n})/\G(\ZZ_p)({\pP}^{n+1})].
\end{equation}
In order to compute the index in the denominator, we first recall the isomorphism $\G(\QQ_p)\cong \su_3(K_{\pP})$ from Lemma \ref{localisom}. Let $\pi$ be a generator of the valuation ideal $\pP$ of $\OO_{\pP}$. As $p\neq 2$, the extension $K_{\pP}/\QQ_p$ is tamely ramified and we can choose $\pi$ such that $\bar \pi=-\pi$, where $\bar \cdot$ denotes the conjugation in $K_{\pP}$ (compare \cite[Prop.~II.5.12]{Lan} ).\\
As a matrix in $M_3(K_{\pP})$, every $g\in\G(\ZZ_p)$ has a $\pP$-adic representation:
\[g=g_0+\pi g_1+\pi^2 g_2+\pi^3 g_3+\ldots,\]
with some matrices $g_j\in M_3(\OO_{\pP}/\pP)=M_3(\ZZ/p\ZZ)$, where the latter equation holds since $p$ is ramified. We note that a matrix $g\in \G(\ZZ_p)$ lies in $\G(\ZZ_p)(\pP^{n})$ if and only if its $\pP$-adic representation starts with
\[g=1_3+\pi^{n} g_n+\pi^{n+1} g_{n+1}+\ldots.\]
For this reason, the application $g\mapsto g_{n}$ gives a map 
\[\G(\ZZ_p)/\G(\ZZ_p)(\pP^n)\longrightarrow M_3(\ZZ/p\ZZ).\] 
We further note that, as $p$ is ramified, the conjugation acts as the identity on $\OO_{\pP}/\pP^{n}$. Now let us consider a representative $g\in \G(\ZZ_p)(\pP^n)$ of the class $[g]\in \G(\ZZ_p)(\pP^n)/\G(\ZZ_p)(\pP^{n+1})$. Two equations characterize $g$ as an element of $\su_3(K_{\pP})$: $g$ is hermitian, that is, $g\bar g^t=(1_3+\pi^n g_n+\ldots)(1_3+\bar \pi^n \bar g_n+\ldots)^t=1_3$ or equivalently $(1_3+\pi^n g_n+\ldots)(1_3- \pi^n \bar g^t_n+\ldots)=1_3$. In other words $g_n-g_n^t=0$, that is, $g_n$ is symmetric. On the other hand, as $\det(g)=\det(1_3+\pi^n g_n+\ldots)=1\equiv 1+\pi^n Tr(g_n)\bmod \pi^{n+1}$, $g_n$ has trace $0$.  By this, the factor group $\G(\ZZ_p)(\pP^n)/\G(\ZZ_p)(\pP^{n+1})$ is characterized as the set of symmetric matrices in $M_3(\ZZ/p\ZZ)$ with trace equal to 0. There are exactly $p^5$ such matrices, and the equation (\ref{pnpn+1}) gives for $n$ odd and $p$ ramified 
\[\G(\ZZ_p)/\G(\ZZ_p)({\pP}^n)=p^{4n-1}(1-p^{-2}),\]
from which the final formula follows.
\end{proof}
Now, summarizing results from Lemma \ref{locind1} and Lemma \ref{locind2}, we get
\begin{theorem}
\label{ind}
Let $\aaa$ be an integral ideal in $K=\QQ(\sqrt{-d})$ and let $\aaa=\pP_1^{e_1}\cdots\pP_t^{e_t}$ be the prime ideal decomposition of $\aaa$ with $p_i=\pP_i\cap\ZZ$. Assume that $p_i\neq 2$ for all $i=1,\ldots t$, if 2 is non-decomposed in $K$. Then it holds 
\[[\Gamma_K:\Gamma_K(\aaa)]=\prod_{i=1}^tp^{\epsilon_i}(1-p_i^{-2})(1-\chi_D(p_i)p_i^{-3}),\]
where $D$ is the discriminant of $K$, $\chi_D(\;)$ is the Jacobi symbol $\chi_D(p)=\left(\frac{D}{p}\right)$ and $\epsilon_i$ is defined by 
\[
 \epsilon_i=\left\{ 
\begin{array}{cl} 
8e_i, & \textrm{if}\ \chi_D(p_i)\neq 0\\
4e_i, & \textrm{if}\ \chi_D(p_i)=0\ \textrm{and}\ e_i\equiv 0\bmod 2\\
4e_i-1, & \textrm{if}\ \chi_D(p_i)=0\ \textrm{and}\ e_i\equiv 1\bmod 2\\
\end{array}\right .
\]
 \end{theorem}
We remark that this result generalizes \cite[Theorem~5A.2.14]{Ho:bsa}. 

\subsection{Determination of neat subgroups}

In this section we discuss the question, for which integral ideals $\aaa$ the principal congruence subgroup $\Gamma_K(\aaa)$ is a neat subgroup. This rather technical property will be used in our consideration of Picard modular surfaces.\\
\begin{lemma}
\label{neat}
Assume that $\aaa$ is an ideal in $\OO_K$ such that $\aaa\cap \ZZ$ and 2 are coprime. Then, the principal congruence subgroup $\Gamma_{K}(\aaa)$ is neat if $N(\aaa)> 3$.
\end{lemma}
\begin{proof}(compare \cite[Lemma 4.3]{Ho:min})\\
Assume that $\Gamma_K(\aaa)$ is not neat. Then, there are eigenvalues $\zeta$ of elements $\gamma\in \Gamma_K(\aaa)$ which generate a torsion subgroup in $\CC^{\ast}$. One concludes that such eigenvalues must be roots of unity (\cite[proof of Lemma 4.3]{Ho:min}). Now, on the one hand, such an eigenvalue $\zeta$ is a root of unity, $n$-th root of unity say, and on the other hand, $\zeta$ is a zero of the characteristic polynomial of $\gamma$ which has degree 3 over $K$. Therefore $[\QQ(\zeta):\QQ]=\varphi(n)\leq 6$. Without loss of generality we can assume that $n=p$ is a prime number. So, in fact we have to check that for $n=2,3,5,7$ the primitive $n$-th root of unity does not appear as the the eigenvalue of elements in $\Gamma_K(\aaa)$ for stated ideals $\aaa$. For this, we first observe that there is a relation between the eigenvalues $\zeta$ of elements in $\Gamma_K(\aaa)$ and the ideal $\aaa$, namely
\[\aaa\ZZ[\zeta]|(\zeta-1)\ZZ[\zeta]\]
Taking norms on both sides we obtain
\[N(\aaa\ZZ[\zeta])=N(\aaa)^{\varphi(n)/2}|N_{\QQ(\zeta)/\QQ}(\zeta-1)=n.\] 
It is easely checked that none of the above primes satisfies this relation by given $\aaa$ with $(2,\aaa\cap\ZZ)=1$ and $N(\aaa)>3$. 
\end{proof}

\section{Proportionality principle for compactified Picard modular surfaces}

As a non-compact almost simple Lie group, $\G(\RR)=\su(h,\CC)$ has a naturally associated irreducible simply connected hermitian symmetric domain $\BB=\su(h)/C$, with a maximal compact subgroup $C$. Taking $h$ to be represented by $\textrm{diag}(1,1,-1)$ this symmetric domain is exactly the 2-dimansional complex unit ball
\[\BB=\BB_2=\{(z_1,z_2)||z_1|^2+|z_2|^2<1\}.\]
It is more natural to consider $\BB$ embedded in $\PP_2(\CC)$ which is interpreted as the compact dual symmetric space of $\BB$. From the construction of $\BB$ it is clear that the group $\su(h)$ acts on $\BB$ as a group of biholomorphic transformations. According to the theorem of Borel and Harish-Chandra (see \cite{BH:1}), the group $\Gamma_K$, and consequently every group $\Gamma$ commensurable to $\Gamma_K$ is a discrete subgroup of finite covolume in the Lie group $\G(\RR)=\su(h,\CC)$ corresponding to $\G$. Therefore, the action of each $\Gamma$ is properly discontinuous, and it is reasonable to consider the quotient space $Y(\Gamma):=\Gamma\backslash \BB$. Now we recall some basic properties of $Y(\Gamma)$

\begin{theorem}
\label{comp}
Let $\Gamma\subset \Gamma_K$ be a Picard modular group and $Y(\Gamma)=\Gamma\backslash \BB$ the corresponding locally symmetric space. Then the following hold:
\begin{enumerate}
\item $Y(\Gamma)$ is not compact.
\item There exists a compactification $\overline{Y(\Gamma)}$ of $Y(\Gamma)$, which set theoreticaly is $Y(\Gamma)\cup\{\textrm{finitely\ many\ points}\}$ and has a structure of a normal projective variety. $\overline{Y(\Gamma)}$ is called the Baily-Borel compactification. 
\item There exists a rational map $\widetilde{Y(\Gamma)}\longrightarrow \overline{Y(\Gamma)}$ such that $\widetilde{Y(\Gamma)}$ is smooth. $\widetilde{Y(\Gamma)}$ is called the smooth compactification of $Y(\Gamma)$.
\end{enumerate}
\end{theorem}

The proof of the above statements, which are in fact special cases of theorems, being valid in much greater generality, can be found in \cite{Ho:bsa}. Let us roughly sketch the main steps of the proofs, also in order to introduce some notions which will be used later.
The non-compactness follows from \cite[11.~8.]{BH:1}. By this theorem $Y(\Gamma)$ is compact if and only if there are no non-trivial unipotent elements in $\G(\QQ)$. But there are many unipotent elements and they all lie in the unipotent radicals of minimal rational parabolic subgroups of $\G$, which themselves fix maximal $h$-isotropic subspaces of $V=K^3$. The compactification $\overline{Y(\Gamma)}$ of $Y(\Gamma)$ is obtained by taking quotient $\overline{Y(\Gamma)}=\Gamma\backslash \BB^{\ast}$, where $\BB^{\ast}=\BB\cup \partial_{\QQ}\BB$ with $\partial_{\QQ}\BB= \G(\QQ)/\Pp$ with some fixed rational parabolic subgroup $\Pp\subset \G$. By a suitable choice of $\Pp$ one can assume that $\partial_{\QQ}\BB=\{[v]\in \PP_2(K)\mid h(v,v)=0\}$. The set $\partial_{\Gamma}\BB=\Gamma\backslash \G(\QQ)/\Pp$ is finite, so $\overline{Y(\Gamma)}=Y(\Gamma)\cup\{\textrm{finitely\ many\ points}\}$. By the theorem of Baily-Borel $\overline{Y(\Gamma)}$ is a normal projective variety. The singularities come from non-trivial torsion elements in $\Gamma$ and possibly from {\bf cusps}, that is, points of $\overline{Y(\Gamma)}-Y(\Gamma)$. If we assume that $\Gamma$ is neat, the theory of toroidal compactifications (see \cite[chapter 4]{Ho:bsa} for the special case of arithmetic lattices in $\su(2,1)$) provides a smooth compactification by replacing every cusp by a smooth elliptic curve. Note that by a theorem of Borel neat normal subgroups $\Gamma^{'}\triangleleft \Gamma_K$ of finite index exist. If $\Gamma$ is not neat, the toroidal compactification provides a surface with at most quotient singularities. They can be resolved by known methods. This resolution gives a map $\widetilde{Y(\Gamma)}\rightarrow \overline{Y(\Gamma)}$, with smooth $\widetilde{Y(\Gamma)}$. 

\subsection{Chern numbers of compactified Picard modular surfaces} 
The two Chern numbers $c_2$ and $c_1^2$ of a smooth compact complex surface are important numerical invariants. The famous proportionality theorem of F.~Hirzebruch states that the Chern numbers of a (2-dimensional) ball quotient by a torsion free cocompact discrete subgroup satisfy $c_1^2=3c_2$. D.~Mumford generalized Hirzebruch's theorem also to non-compact quotients. Before we state the proportionality result for Picard modular groups, as described in \cite{Ho:bsa}, we need to introduce some notions.\\
Let $\Gamma$ be a neat Picard modular group. The quotient $Y(\Gamma)$ is not compact, but one can still define Chern numbers of $Y(\Gamma)$ as the volumes
\begin{align*}
c_2(Y(\Gamma))&=\int_{F(\Gamma)}\eta_2\\
c_1^2(Y(\Gamma))&=\int_{F(\Gamma)}\eta_1^2
\end{align*}
where $F(\Gamma)$ is a fundamental domain of $\Gamma$ in $\BB$ and $\eta_2$ and $\eta_1^2$ are suitably normalized volume forms related to the $\Gamma$-invariant Bergman metric on $\BB$.\\
On the other hand, on the smooth compactification $\widetilde{Y(\Gamma)}$ of $Y(\Gamma)$ there are also usual Chern numbers. Before we relate these two kind of Chern numbers, we note that $\widetilde{Y(\Gamma)}$ arises from $Y(\Gamma)$ by replacing a cusp $\kappa\in \partial_{\Gamma}\BB$ by an elliptic curve $E_{\kappa}$. Hence, the difference between $\widetilde{Y(\Gamma)}$ and $Y(\Gamma)$ is encoded in the {\bf compactification divisor} $T_{\Gamma}=\sum_{\kappa\in \partial_{\Gamma}\BB} E_{\kappa}$. 
\begin{proposition}(\cite[Proposition 4.3.6]{Ho:bsa})
\label{invpic}
Let $\Gamma$ be a neat Picard modular group, $Y(\Gamma)=\Gamma\backslash \BB$, $\widetilde{Y(\Gamma)}$ the smooth compactification of $Y(\Gamma)$ and $T_{\Gamma}$ the compactification divisor. Let $(T_{\Gamma}.T_{\Gamma})$ denote the selfintersection number of $T_{\Gamma}$. Then:
\begin{enumerate}
\item $c_2(\widetilde{Y(\Gamma)})=c_2(Y(\Gamma))$
\item $c_1^2(\widetilde{Y(\Gamma)})=3 c_2(Y(\Gamma))+(T_{\Gamma}.T_{\Gamma})$
\end{enumerate}

\end{proposition}
 
\subsection{Chern numbers of quotients by principal congruence subgroups}
By the above propostion \ref{invpic}, two magnitudes are essential for the computation of Chern invariants of the smooth compactification of a Picard modular surface $Y(\Gamma)=\Gamma\backslash \BB$, namely the volume $c_2(Y(\Gamma))$ and the selfintersection number $(T_{\Gamma}.T_{\Gamma})$. In the case of a ball quotient by a principal congruence subgroup this two numbers can be given in terms of a number theoretic expression. Let us shortly write $Y(\aaa)$ for the ball quotient $\Gamma_K(\aaa)\backslash \BB$ by the principal congruence subgroup $\Gamma_K(\aaa)$. In this context we also write $Y(1)$ for the quotient $\Gamma_K\backslash \BB$ by a full Picard modular group.
\begin{proposition}(see \cite[Thm.~5A.4.7]{Ho:bsa})
\label{vol}
Let $K=\QQ(\sqrt{-d})$ and let $\Gamma_{K}$ be the full Picard modular group corresponding to $K$. Let $D$ be the discriminant of $K$ and $\chi_D(\;)=\left ( \frac{D}{\cdot}\right )$ the Dirichlet character associated to $K$. Then
\[c_2(Y(1))=c_2(\Gamma_K\backslash \BB)=\delta_K\frac{|D|^{5/2}}{32\pi^3}L(3,\chi_D),\]
where $L(s,\chi_D)$ denotes the Dirichlet L-function associated with $\chi_D$, and $\delta_K$ is defined as $\delta_K=1/3$ if $d=3$, and $\delta_K=1$ otherwise.\\
Consequently, for every integral ideal $\aaa$ of $K$ and the principal congruence subgroup $\Gamma_K(\aaa)$ 
\[c_2(Y(\aaa))=[\Gamma_K:\Gamma_K(\aaa)]\delta_K\frac{|D|^{5/2}}{32\pi^3}L(3,\chi_D).\] 
\end{proposition}
\begin{remark}
The above volume formula can be seen as a special case of a general formula developed by G.~Prasad, expressing the volume of a quotient of a Lie group by an arithmetic subgroup, by values of L-functions. The reader can consult \cite{PY} for the case of the unitary group corresponding to a hermitian form of signature $(2,1)$.
\end{remark}

A similar arithmetic expression for the selfintersection number of the compactification divisor can be deduced using 
\cite[Lemma 4.7, Lemma 4.8]{Ho:min}: 
\begin{proposition}
\label{self}
Let $h_K$ denote the ideal class number of $K$ and let $\eta_K$ and $\vartheta_{\aaa}$ be defined as
 \[\vartheta_{\aaa}=\min \{n\in\NN\mid \aaa|n\sqrt{-d}\},\]
\[
\eta_K=\left\{
\begin{array}{ll} 
1 & \textrm{if}\ D=-4\\
1/6 & \textrm{if}\ D=-3\\
2 & \textrm{if}\ D\neq-4,D\equiv 0\bmod 4\\
1/2 & \textrm{if}\ D\neq -3, D\equiv 1\bmod 4\\
\end{array} \right.
\]
Then, we have the following formula for the selfintersection number $(T_{\Gamma_K(\aaa)}.T_{\Gamma_K(\aaa)})$:
\[(T_{\Gamma_K(\aaa)}.T_{\Gamma_K(\aaa)})=-\frac{h_K\eta_K}{\vartheta_{\aaa}^2}[\Gamma_K:\Gamma_K(\aaa)]. \]
\end{proposition}

\subsection{Values of Dirichlet L-functions at integers}
\label{bernoulli}
In order to get concrete values of Chern numbers of Picard modular surfaces, we need to know the value $L(3,\chi_D)$ explicitely. Let us briefly discuss the well-known method for the computation of this number.\\

We first recall the functional equation for Dirichlet L-functions (see, for instance, \cite{Wash}, chapter 4, in particular Thm.~4.2). This equation implies 

\[L(3,\chi_D)=-\frac{2\pi^3}{|D|^{5/2}}L(-2,\chi_D)=\frac{2\pi^3}{3|D|^{5/2}}B_{3,\chi_D}.\]

There, $B_{n,\chi_D}$ denotes the $n$-th generalized Bernoulli number associated with $\chi_D$ (see \cite[chapter 4]{Wash} for the definition). The generalized Bernoulli numbers can be computed easily by the following formula, which for instance can be found in \cite[Proposition 4.2]{Wash}:
\begin{equation}
\label{bernf}
B_{n,\chi_D}=|D|^{n-1}\sum_{k=1}^{|D|}\chi_D(k)B_n(\frac{k}{|D|}).
\end{equation}
In the above formula $B_n(X)$ denotes the $n$-th Bernoulli polynomial which can be defined as 
\[B_n(X)=\sum_{k=0}^n \binom{n}{k}B_kX^{n-k},\]
$B_k$ denoting the usual $k$-th Bernoulli number. It is now easy to provide a list of values $L(3,\chi_D)$, knowing $B_3(X)=\frac{1}{2}X-\frac{3}{2}X^2+X^3$. Here we give the first values of $B_{3,\chi_D}$:
\[
\begin{array}{c|c|c|c|c|c|c|c|c|c|c|c|c}
 |D| & 3 & 4 & 7 & 8 & 11 & 15 & 19 & 20 & 23 & 24 & 31 & 35\\
\hline
B_{3,\chi_D} & 2/3 & 3/2 & 48/7 & 9 & 18 & 48 & 66 & 90 & 144 & 138 & 288 & 108
\end{array}
\]
 
\section{Applications}
In this section we discuss some applications of the explicit formulas for the Chern invariants of compactified ball quotients by principal congruence subgroups. 

\subsection{Classification}

\begin{theorem}
Let $\aaa$ be an integral ideal of $K=\QQ(\sqrt{-d})$ such that $N(\aaa)>3$ and $\aaa\cap \ZZ$ and the prime 2 are  coprime. The smooth compactification $\widetilde{Y(\aaa)}=\widetilde{\Gamma_K(\aaa)\backslash\BB}$ is a surface of general type, with the (additional) possible exception for $K=\QQ(\sqrt{-7})$ and $\aaa\cap\ZZ$ ramified in $K$.
\end{theorem}
\begin{proof}
First we note that $\widetilde{Y(\aaa)}$ is smooth and compact algebraic surface, since $\Gamma_K(\aaa)$ is neat by Lemma \ref{neat} (see also Theorem.~\ref{comp}). In order to show that $\widetilde{Y(\aaa)}$ is of general type it is sufficient to check that $c_1^2(\widetilde{Y(\aaa)})>9$. This follows from the classification theory of algebraic surfaces. It is also sufficient to check the inequality only for the surfaces $\widetilde{Y(\pP)}$, where $\pP$ is a prime ideal, since the Kodaira dimension can not decrease by the passage to a finite covering. Hence, if $\widetilde{Y(\pP)}$ is of general type, then also $\widetilde{Y(\aaa)}$ for any ideal $\aaa$ such that $\pP|\aaa$.\\
From Lemma \ref{invpic}, Proposition \ref{vol}, and Proposition \ref{self}, we obtain the following formula for $c_1^2(\widetilde{Y(\pP)})$:
\begin{equation}
\label{c12}
c_1^2(\widetilde{Y(\pP)})=[\Gamma_K:\Gamma_K(\aaa)]\left (\frac{3\delta_K|D|^{5/2}}{32\pi^3}L(3,\chi_D)-\frac{h_K\eta_K}{\vartheta_{\pP}^2}\right )
\end{equation}
Looking at the index formula in Theorem \ref{ind}, we see that $[\Gamma_K:\Gamma_K(\pP)]>9$ if $N(\pP)>3$ and $\pP\cap \ZZ\neq (2)$. Thus, it suffices to show that 
\begin{equation}
\label{schlgl}
\frac{3\delta_K|D|^{5/2}}{32\pi^3}L(3,\chi_D)-\frac{h_K\eta_K}{\vartheta_{\pP}^2}\geq 1
\end{equation}
First we show that the inequality (\ref{schlgl}) holds, if $|D|> 35$.\\
Recall first the analytic class number formula 
\[
h_K=\frac{\mu_K\sqrt{-D}}{2\pi}L(1,\chi_D)=-(\mu_K/4)B_{1,\chi_D},
\] 
where $\mu_K$ denotes the number of roots of unity contained in $K$. Using the facts on generalized Bernoulli numbers summarized in section \ref{bernoulli}, and in particular the formula (\ref{bernf}), we get a trivial bound $h_K\leq |D|/4$ which holds for $|D|\neq 3,4$. On the other hand, comparing the Euler factors the value $L(3,\chi_D)$ can be estimated below by a known value of the Riemann zeta function, namely $L(3,\chi_D)>1/\sqrt{\zeta_{\QQ}(2)}=\sqrt{6}/\pi$. More precisely, one shows that $(1-p^{-2})(1-\chi_D(p)p^{-3})^2<1$ for all primes, from which the lower bound follows. Noting additionally that always $\eta_K\leq 2$ and $\vartheta_{\pP}\geq 1$, the inequality (\ref{schlgl}) will follow from the inequality 
\begin{equation}
 \label{schlgl2}
|D|^{5/2}>\frac{16\pi^4}{3\sqrt{6}}(2+|D|)
\end{equation}
This inequality holds for $|D|> 35$.\\
For discriminants $D$ with $|D|\leq 35$ we have to check the inequality $c_1^2(\widetilde{Y(\pP)})>9$ case by case. For this, one first uses the functional equation to reformulate the formula (\ref{c12}) as 
\[
c_1^2(\widetilde{Y(\pP)})=[\Gamma_K:\Gamma_K(\aaa)]\left (\frac{3\delta_K}{16}B_{3,\chi_D}-\frac{h_K\eta_K}{\vartheta_{\pP}^2}\right )
 \]
and than uses the exact values $B_{3,\chi_D}$ from the table in section \ref{bernoulli}. For instance, let us consider the case $D=-35$. With $h_K=2$, $\eta_K=1/2$ and $B_3,\chi_{-35}=108$ we have $c_1^2=[\Gamma_K:\Gamma_K(\pP)](81/4-1/\vartheta_{\pP}^2)\geq 81 [\Gamma_K:\Gamma_K(\pP)]/4 > 9$. The other cases 
are treated in the same manner. The cases $|D|=-3,-4,-7,-8$ need to be considered more carefully, since in these situations $c_1^2$ can be negative.\\
$D=-8$. Let us shortly write $I_{\pP}$ instead of $[\Gamma_K:\Gamma_K(\pP)]$. For the selfintersection we have: $c_1^2=I_{\pP}\{27/48-2/\vartheta_{\pP}^2\}$. Now, it is easy to see that $c_1^2< 0$ if and only if $\vartheta_{\pP}=1$ which is exactly the case when $p=\pP\cap\ZZ$ is ramified in $K=\QQ(\sqrt{-2})$. From this it follows that $p=2$, and this case is excluded. Otherwise, $c_1^2>9$ if $I_{\pP}> 144$. This is the case for $N(\pP)>3$ and $p$ unramified, as the index formula shows.\\
$D=-7$. Here we have $c_1^2=I_{\pP}\{3/7-1/2\vartheta^2_{\pP}\}$. This expression is negative if and only if $\vartheta_{\pP}=1$, which means that $p=\pP\cap\ZZ$ is ramified. This is the exeption we have to exclude. Indeed, for $\pP=(\sqrt{-7})$ we have $c_2=48$ and $c_1^2=-24$. In other cases $c_1^2>9$.\\
$D=-4$. Here we get a negative $c_1^2$ for $\vartheta_{\pP}=1$ or $2$. But $\vartheta_{\pP}=1$ is only possible for $\pP=(1)$ and $\vartheta_{\pP}=2$ is only possible for those $\pP$ which contain the prime $2$.\\
$D=-3$. In this case $c_1^2<0$ is only possible for a prime ideal $\pP$ such that $p$ is ramified in $K$. But then $N(\pP)=3$, a case which is excluded.\\ 
\end{proof}

The invariants of $\widetilde{Y(\aaa)}$ are in general very large due to the fact that the index, which grows very fast with the norm of the ideal, dominates the expressions. However, the surfaces $\widetilde{Y(\aaa)}$ are interesting from the point of view of surface geography, since the pairs $(c_1^2(\widetilde{Y(\aaa)}),c_2(\widetilde{Y(\aaa)}))$ are located in an interesting region of the $(c_1^2,c_2)$-plane, namely the ratio $c_1^2/c_2$ is close to $3$ (see \cite{Ho:min}, and first part of \cite{hunt} for a detailed discussion on surface geography). More precisely the ratio $c_1^2(\widetilde{Y(\aaa)})/c_2(\widetilde{Y(\aaa)})$ tends to $3$ when either the norm $N(\aaa)$ or the discriminant $|D|$ tends to infinity (see also \cite{Ho:min}.)\\ 
This is easily seen, since 
\begin{equation}
\label{ratio}
\frac{c_1^2(\widetilde{Y(\aaa)})}{c_2(\widetilde{Y(\aaa)})}=3-\frac{\frac{h_K\eta_K}{\vartheta_{\aaa}^2}}{\frac{|D|^{5/2}\delta_K}{32\pi^3}L(3,\chi_D)}
 \end{equation}

Now, as $\vartheta_{\aaa}^2\longrightarrow \infty$ for $N(\aaa)\rightarrow \infty$, we see that for a fixed $D$ the ratio tends to 3. Let us, on the other hand, fix $\aaa$ but let $D$ vary. We first observe, as $L(3,\chi_D)$ is bounded by the value of the Riemann zeta function at $s=3$, namely $\zeta(3)^{-1}<L(3,\chi_D)<\zeta(3)$, that there are constants $C$ and $C^{'}$ not depending on $D$ such that 
\begin{equation}
\label{ineq}
C^{'}\frac{h_K}{|D|^{5/2}}\geq \frac{\frac{h_K\eta_K}{\vartheta_{\aaa}^2}}{\frac{|D|^{5/2}\delta_K}{32\pi^3}L(3,\chi_D)}\geq C\frac{h_K}{|D|^{5/2}}
\end{equation}  
There are well known bounds for the class number of an imaginary quadratic field, namely for any $\epsilon>0$ there are constants $C_1=C_1(\epsilon)$ and $C_2=C_2(\epsilon)$ such that $C_1|D|^{1/2+\epsilon}>h_K >C_2|D|^{1/2-\epsilon}$. Bringing this into (\ref{ineq}) we see that the expression $\frac{h_K\eta_K}{\vartheta_{\aaa}^2}/\frac{|D|^{5/2}\delta_K}{32\pi^3}L(3,\chi_D)$ tends to zero as $|D|\rightarrow \infty$. Hence
\[\lim_{|D|\rightarrow \infty}\frac{c_1^2(\widetilde{Y(\aaa)})}{c_2(\widetilde{Y(\aaa)})}=3\] 
Already for small discriminants and norms, the Chern number ratio is very close to $3$, see \cite{Ho:min}, table 1.

\subsection{Dimension formulas} 
Another application of the above explicit formulas for the Chern numbers of Picard modular surfaces is an explicit formula for the dimensions of spaces of cusp forms relative to a Picard modular group $\Gamma$.\\
A holomorphic function $f$ on the ball is called a {\bf cusp form of weight} $k$ {\bf with respect to} $\Gamma$ if it satisfies the following conditions:
\begin{itemize}
\item For all $\gamma\in \Gamma$ and all $z\in \BB$ $f(\gamma z)=j(\gamma,z)^{k}f(z)$, where $j(\gamma,z)$ denotes the determinant of the Jacobian matrix corresponding to the holomorphic map $\gamma:\BB\longrightarrow \BB$ at the point $z$.
\item $f(z)$ vanishes at cusps of $\Gamma$, i.~e.~it vanishes on $\partial_{\Gamma}\BB=\Gamma\backslash\partial_{\QQ}\BB$. 
\end{itemize}
Let $\Gamma$ be a neat Picard modular group. In the following, $S_k(\Gamma)$ will denote the space of cusp forms of weight $k$ and with respect to $\Gamma$. By a result of Hemperley \cite{Hem}, the cusp forms of weight $k$ can be interpreted as the sections of a certain line bundle $\LL^{(k)}$ on the smooth compactification $\widetilde{Y(\Gamma)}$. Thus, $\dim S_k(\Gamma)=\dim H^0(\widetilde{Y(\Gamma)},\LL^{(k)})$. As an application of the Riemann-Roch and the Kodaira vanishing theorem, one gets the following formula for $\dim S_k(\Gamma)$
\begin{theorem}[see \cite{Ho:zetadim}]
Let $\Gamma$ be a neat Picard modular group and $k\geq 2$.
\[
\dim S_k(\Gamma)=[\Gamma_K:\Gamma]\left\{ \frac{9k(k-1)+2}{6}c_2(Y(\Gamma))+\frac{1}{12}(T_{\Gamma}.T_{\Gamma})\right\}
\] 
\end{theorem}
Now it is clear that we have an explicit formula for the dimensions of spaces of cusp forms of weight $k\geq 2$ which are automorphic with respect to a neat principal congruence subgroup $\Gamma_K(\aaa)$, namely

\[
S_k(\Gamma_K(\aaa))=\frac{[\Gamma_K:\Gamma_K(\aaa)]}{6}\left\{ (9k(k-1)+2)\frac{\delta_K|D^{5/2}|}{32\pi^3}L(3,\chi_D)-\frac{h_K\eta_K}{2\vartheta_{\aaa}^2}\right\}
\]

These dimensions are very large in general, even in the case of small discriminants and small norms.

\bibliography{Refs_Invariants_of_Picard_modular_surfaces}
\bibliographystyle{alpha}

\end{document}